\documentclass[11pt]{article}

\usepackage{amsfonts,amssymb}

\usepackage{amsthm}

\usepackage{amsmath}

\usepackage{verbatim}

\theoremstyle{plain}

\newtheorem{theorem}{Theorem}[section]
\newtheorem{lemma}[theorem]{Lemma}
\newtheorem{proposition}[theorem]{Proposition}
\newtheorem{corollary}[theorem]{Corollary}
\newtheorem{Counter-example}[theorem]{Counter-example}
\newtheorem{remark}[theorem]{Remark}

\theoremstyle{definition}

\theoremstyle{remark}

\long\def\symbolfootnote[#1]#2{\begingroup\def\thefootnote{\fnsymbol{footnote}}
\footnote[#1]{#2}\endgroup}

\usepackage{amssymb}

\begin{document}

\def\Q{\mathbb Q}
\def\R{\mathbb R}
\def\N{\mathbb N}
\def\Z{\mathbb Z}
\def\C{\mathbb C}
\def\S{\mathbb S}
\def\L{\mathbb L}
\def\H{\mathbb H}
\def\K{\mathbb K}
\def\X{\mathbb X}
\def\Y{\mathbb Y}
\def\Z{\mathbb Z}
\def\E{\mathbb E}
\def\J{\mathbb J}
\def\I{\mathbb I}
\def\T{\mathbb T}
\def\H{\mathbb H}

\title{On spacelike surfaces in $4$-dimensional Lorentz-Minkowski spacetime \\ through a lightcone 
}
\author{Francisco J. Palomo 
 \\ Departamento de Matem\'{a}tica
Aplicada \\
Universidad de M\'{a}laga \\ 29071-M\'{a}laga (Spain)\\
{\ttfamily fjpalomo@ctima.uma.es}
\and
Alfonso Romero \\Departamento de Geometr\'{\i}a y Topolog\'{\i}a \\
Universidad de Granada \\ 18071-Granada (Spain) \\ {\ttfamily
aromero@ugr.es}
}

\date{}

\maketitle

\symbolfootnote[ 0 ]{Partially supported by the Spanish MEC-FEDER Grant MTM2010-18099
and the Junta de Andalucia Regional Grant P09-FQM-4496 with FEDER
funds.
}

\thispagestyle{empty}

\begin{abstract}
On any spacelike surface in a lightcone of four dimensional Lorentz-Minkowski space a distinguished smooth function is considered. It is shown how both extrinsic and intrinsic geometry of such a surface is codified by this function. The existence of a local maximum is assumed to decide when the spacelike surface must be totally umbilical, deriving a Liebmann type result. Two remarkable families of examples of spacelike surfaces in a lightcone are explicitly constructed. Finally, several results which involve the first eigenvalue of the Laplace operator of a compact spacelike surface in a lightcone are obtained.

\end{abstract}

\section{Introduction}

The classical Liebmann theorem states that the only compact surfaces with constant Gauss curvature of the Euclidean three dimensional space are the totally umbilical round spheres. This property is also satisfied if the Euclidean space is replaced by the three dimensional hyperbolic space $\H^{3}$ or for an open hemisphere of the three dimensional sphere (see for instance \cite{AAR}). 

In the Lorentzian setting, the result remains true for compact spacelike surfaces in the three dimensional De Sitter space $\S^{3}_{1}$ \cite{AR}. Note that $\H^{3}$ and $\S^{3}_{1}$ can be seen as hyperquadrics of the four dimensional Lorentz-Minkowski space $\L^{4}$. From this point of view, the Liebmann theorems for $\H^{3}$ and $\S^{3}_{1}$ could be enunciated as follows. Every compact spacelike surface in $\L^{4}$ through the hyperquadrics $\H^{3}$ and $\S^{3}_{1}$ with constant Gauss curvature is a totally umbilical round sphere. Besides  of $\H^{3}$ and $\S^{3}_{1}$, there is another hyperquadric in $\L^{4}$ with relevant geometry: the lightcone. Recall that the induced metric from $\L^{4}$ on a lightcone is degenerate. However, among the spacelike surfaces in $\L^{4}$, those that lie in a lightcone constitute an outstanding class.

In fact, any simply-connected two dimensional Riemannian manifold can be isometrically immersed in a lightcone of $\L^4$ \cite{LUY}. Thereofore, any point of an arbitrary two dimensional Riemann manifold $(M^{2},g)$ has a neighbourhood that can be isometrically immersed in a lightcone. Alternatively, this fact also follows from the local existence of isothermal parameters and the argument at the begining of Section 4. Hence, we have no local intrinsic information for $(M^{2},g)$ if it admits an isometric immersion in a lightcone. On the contrary, the higher dimensional case goes in a different direction. In fact, an $n(\geq 3)$-dimensional Riemannian manifold is conformally flat (i.e., any of each points lies in a neighborhood which is conformally equivalent to an open subset of the $n-$dimensional Euclidean space, with its canonical metric) if and only if it can be locally isometrically immersed in a lightcone of the $(n+2)$-dimensional Lorentz-Minkowski space $\L^{n+2}$ \cite{AD}, \cite[Cor. 7.6]{Da}.

It is not difficult to see that a compact spacelike surface in a lightcone is diffeomorphic to a two dimensional sphere $\S^{2}$ (Proposition \ref{esf}). This is also the case for a compact spacelike surface of $\S^{3}_{1}$ \cite{AR}. However, whereas any Riemannian metric on $\S^{2}$ can be obtained from a spacelike immersion in a lightcone of $\L^{4}$, 
it is known that if a Riemannian metric $g$ on  $\S^{2}$ has Gauss curvature $K>1$, then there is no isometric immersion from $(\S^{2},g)$ in the (unit) De Sitter space $\S^{3}_{1}$ \cite[Cor. 10]{AR}.

In this paper, we will mainly focus on the global geometry of spacelike surfaces in $\L^4$ through a lightcone. Thus, the following problem arises in a natural way:
\begin{quote}
{\it Is a complete spacelike immersion of constant Gauss curvature in a lightcone totally umbilical in $\L^4$? 
}
\end{quote}
An answer to this question is given in Theorem \ref{complejo} where it is shown that, under the assumptions of completeness and the existence of a local maximum of certain smooth function, the spacelike surface must a totally umbilical round sphere.

This paper is organized as follows. First, Section 2 gives the basic background formulae. For any spacelike orientable surface $M^2$ in $\L^4$, two
independent normal lightlike vector fields $\xi$ and $\eta$ are introduced. We then give expresions, in terms of $\xi$ and $\eta$, for the Gauss curvature $K$ of $M^2$, (\ref{100}), and for the mean curvature vector field $\mathbf{H}$, (\ref{curme}).

The local geometry of spacelike surfaces through a lightcone is studied in Section 3. We begin showing a characterization of spacelike surfaces $M^2$ of $\L^4$ in a lightcone, Proposition \ref{characterization}. Then, we explicity construct, in this case, the lightlike normal vector fields $\xi$ and $\eta$ (Lemma \ref{campos}) and compute the corresponding Weingarten operators. Next, the mean curvature vector field and the Gauss curvature are shown. It is a remarkable fact that for any surface in a lightcone, its Gauss curvature and mean curvature vector field are related by (\ref{Gauss_curvature2}).
$$
K=\langle \mathbf{H},\mathbf{H}\rangle.
$$
This formula shows an interesting relation between intrinsic and extrinsic geometry of a spacelike surface in a lightcone and has a nice consequence in the compact case. In fact, from Proposition \ref{esf} and making use of the Gauss-Bonnet theorem, we get that the Willmore integral on any compact spacelike surface in a lightcone is constant, independently of the spacelike immersion (Remark \ref{will}).

Section 4 deals with the construction of two remarkable families of examples of spacelike surfaces in a lightcone.
In order to do that, recall that the lightcone at the origin is invariant under conformal transformations. Therefore, for every spacelike immersion $\psi:M^2\to \L^4$ through the lightcone at the origin and every smooth function $\sigma$ on $M^2$, a new spacelike immersion is constructed by $\psi_{\sigma}=e^{\sigma}\psi$. The first family comes from an isometric immersion $\psi$ of the Euclidean plane $\E^2$ into the future lightcone $\Lambda^{+}$, and the second one from an isometric immersion $\psi$ of the unit $2$-dimensional sphere into $\Lambda^{+}$. In the noncompact case, the totally umbilicity of $\psi_{\sigma}$ 
is described by mean a partial differential system (\ref{PDE}). Several particular solutions are then listed. In the compact case, all totally umbilical spacelike immersions of $\S^2$ in a lightcone are explicitly constructed.

Finally, we deal with the first eigenvalue $\lambda_{1}$ of the Laplace operator of a compact spacelike surface $M^{2}$ in a lightcone. We point out that in this case we have (\ref{Re}),
$$
\lambda_{1}\leq 2 \,\frac{\int_{M^{2}}\langle \mathbf{H},\mathbf{H}\rangle\,dA}{\mathrm{area}(M^2)}.
$$
That is, the Reilly extrinsic bound of $\lambda_{1}$ \cite{Re} for a compact surface in $m$-Euclidean space holds true in this setting. However, this is not true for a general compact spacelike surface in $\L^{4}$ (see Remark \ref{Reilly}). Using this inequality for $\lambda_{1}$, we are able to characterize the totally umbilical round spheres in a lightcone (Theorem \ref{calor}). Moreover, a comparison area result is obtained which also get another distinguishing property of totally umbilical round spheres (Proposition \ref{area}).

\hyphenation{Lo-rent-zi-an}
\section{Preliminaries}

Let $\L^4$ be Lorentz-Minkowski spacetime, that is, $\R^4$ endowed
with the Lorentzian metric tensor $$\langle\, ,\, \rangle =
-(dx_{0})^{2}+(dx_{1})^{2}+(dx_{2})^{2}+(dx_{3})^{2},$$ where
$(x_{0},x_{1},x_{2}, x_{3})$ are the canonical coordinates of
$\R^{4}$. A smooth immersion $\psi:M^{2}\rightarrow \L^4$ of a
$2$-dimensional (connected) manifold $M^2$ is said to be
spacelike if the induced metric tensor via $\psi$ (denoted also by
$\langle\, ,\, \rangle$) is a Riemannian metric on $M^2$. In this
case, we call $M^2$ as a spacelike surface.

Let $\nabla$ and $\overline{\nabla}$ be the Levi-Civita
connections of $M^2$ and $\L^4$, respectively, and let
$\nabla^{\perp}$ be the connection on the normal bundle. The Gauss
and Weingarten formulas are
$$\overline{\nabla}_X Y=\psi_{*}(\nabla_XY) + \mathrm{II}(X,Y)
\, \quad \mathrm{and} \, \quad
\overline{\nabla}_X\xi=-\psi_{*}(A_{\xi}X)+\nabla^{\perp}_X\,\xi,$$
for any $X,Y \in \mathfrak{X}(M^{2})$ and $\xi \in
\mathfrak{X}^{\perp}(M^{2})$, and where $\mathrm{II}$ denotes the
second fundamental form of $\psi$. The shape (or Weingarten) operator
corresponding to $\xi$, $A_{\xi}$, is related to $\mathrm{II}$ by
$$\langle A_{\xi}X, Y \rangle = \langle \mathrm{II}(X,Y), \xi
\rangle,$$
for all $X,Y \in \mathfrak{X}(M^{2})$. The mean curvature vector field of $\psi$ is given by
$\mathbf{H}=\frac{1}{2}\,\mathrm{tr}_{_{\langle\; , \;
\rangle}}\mathrm{II},$ and the Gauss and Codazzi equations are
respectively,
\begin{equation}\label{1}
R(X,Y)Z= A_{_{\mathrm{II}(Y,Z)}}X-A_{_{\mathrm{II}(X,Z)}}Y
\end{equation}

\begin{equation}\label{2}
(\widetilde{\nabla}_X\mathrm{II})(Y,Z)=(\widetilde{\nabla}_Y\mathrm{II})(X,Z),
\end{equation}
where $R$  stands for
the curvature tensor\footnote{Our convention on the sign is
$R(X,Y)Z=\nabla_{X}
\nabla_{Y}Z-\nabla_{Y}\nabla_{X}Z-\nabla_{[X,Y]}Z$.} of the induced metric and
$$(\widetilde{\nabla}_X\mathrm{II})(Y,Z)=\nabla^{\perp}_X\,\mathrm{II}(Y,Z)-\mathrm{II}(\nabla_XY,Z)-\mathrm{II}(Y,\nabla_XZ),$$
for any $X,Y,Z \in \mathfrak{X}(M^2)$. For each $\xi\in
\mathfrak{X}^{\perp}(M^{2})$, the Codazzi equation gives,
\begin{equation}\label{Cod}
(\nabla_{X}A_{\xi})Y-(\nabla_{Y}A_{\xi})X=
A_{\nabla^{\perp}_{X}\xi}Y - A_{\nabla^{\perp}_{Y}\xi}X.
\end{equation}
We denote by $\mathrm{II}_{\xi}$ the symmetric tensor field on $M^2$ defined
by,
$$\mathrm{II}_{\xi}(X,Y)=-\langle A_{\xi}X, Y \rangle=-\langle \mathrm{II}(X,Y), \xi \rangle.$$

\begin{remark}{\rm In contrast with the case of spacelike hypersurface in $\L^{4}$, a spacelike surface $M^2$ in $\L^{4}$ may be nonorientable.}
\end{remark}

If we assume $M^{2}$ is orientable, then we can globally take two
independent lightlike normal vector fields $\xi, \eta\in
\mathfrak{X}^{\perp}(M^{2})$ with $\langle \xi,\eta\rangle=1$, and the
following (global) formula holds,
\begin{equation}
\mathrm{II}(X,Y)=-\mathrm{II}_{\eta}(X,Y)\,\xi-\mathrm{II}_{\xi}(X,Y)\,\eta,
\end{equation}
for every $X,Y\in \mathfrak{X}(M^{2})$. Therefore,
\begin{equation}\label{curme}
\mathbf{H}=\frac{1}{2}\,(\mathrm{tr}
A_{\eta})\,\xi+\frac{1}{2}\,(\mathrm{tr} A_{\xi})\,\eta.
\end{equation}
Contracting (\ref{1}) we obtain,
\begin{equation}\label{3}
\mathrm{Ric}(Y,Z)=2\,\langle A_{\mathbf{H}}Y,Z\rangle-\langle
\,(A_{\xi}A_{\eta}+ A_{\eta}A_{\xi}\,)Y, Z \rangle,
\end{equation}
that is,
\begin{equation}\label{100}
K\, I=2\,A_{\mathbf{H}}-A_{\xi}A_{\eta}-A_{\eta}A_{\xi},
\end{equation}
where $K$ is the Gauss curvature of $M^{2}$ and $I$ the identity transformation.
Therefore, we get
\begin{equation}\label{101}
2\,K=4\,\langle \mathbf{H},\mathbf{H}
\rangle-2\,\mathrm{tr}(A_{\xi}A_{\eta})=4\,\langle
\mathbf{H},\mathbf{H} \rangle-\langle \mathrm{II},\mathrm{II}
\rangle,
\end{equation}
where, $\langle
\mathrm{II},\mathrm{II}\rangle_{q}=\Sigma_{i,j=1}^{2}\langle
\mathrm{II}(e_{i},e_{j}),\mathrm{II}(e_{i},e_{j})\rangle$, for an orthonormal basis $\{e_{1},e_{2}\}$ of $T_{q}M^{2}$, $q\in M^{2}$ is the squared lenght of the second fundamental form.

\section{Local geometry of a spacelike surface in a lightcone}

We write $$\Lambda^+ = \{\,v\in \L^{4}\,:\,\langle v,v \rangle
=0,\, v_{0}>0\,\}$$ for the future lightcone of $\L ^4$. For every
$p\in \L^4$, the future (resp. past) lightcone at $p$ is given by
$ \Lambda^{+}(p)=p+\Lambda^+ $ (resp.
$\Lambda^{-}(p)=p-\Lambda^+$). A spacelike surface $\psi:M^{2}\rightarrow \L^4$ factors through a
lightcone at $p\in \L^4$ if $\psi(M^{2})\subset \Lambda^{+}(p)$ or
$\psi(M^{2})\subset \Lambda^{-}(p)$.

We begin recalling the following result \cite[0.9]{Mag} which characterizes when
a spacelike surface in $\L^{4}$ lies in some lightcone (compare with \cite[Th. 4.3]{IPR}).

\begin{proposition}\label{characterization}
Let $\psi:M^{2}\rightarrow \L^4$ be a spacelike surface. Then the
following conditions are equivalent.
\begin{enumerate}
\item[{\rm(i)}] The immersion $\psi$ factors through a lightcone.
\item[{\rm(ii)}] There exist a lightlike normal vector field $\overline{\xi}\in
\mathfrak{X}^{\perp}(M^{2})$ and $\lambda \in C^{\infty}(M^2)$,
$\lambda>0$, such that
$$A_{\overline{\xi}}=-\lambda\, I \quad \mathrm{and} \quad \nabla^{\perp}\overline{\xi}=d(\log \lambda)\overline{\xi}.$$
\item[{\rm(iii)}] There exists a lightlike normal vector field $\xi$,
parallel with respect to the normal connection, and such that
$A_{\xi}=-Id$.
\end{enumerate}
\end{proposition}

After a suitable translation, we always assume that a spacelike surface in a lightcone is contained in a lightcone at the origin. Morever, there is no lost of generality to assume that this lightcone is $\Lambda^{+}$.

In that follows, we put $\partial_{0}=\partial/\partial x_0$, $\partial_{0}\circ \psi $ the vector field $\partial_{0}$ along the immersion $\psi$,
$\psi_0$ the first component of $\psi$ and $\nabla$ the gradient operator of $M^2$.
\begin{lemma}\label{campos}
Let $\psi:M^{2}\rightarrow \L^4$ be a spacelike surface which
factors through the lightcone at $\Lambda^{+}$. Then,
$$
\xi=\psi\,\,\quad \textrm{and}\,\,\quad \eta=\frac{1+\| \nabla
\psi_{0}\|^2 }{2\,\psi_{0}^2}\, \xi -
\frac{1}{\psi_{0}}\,\Big(\partial_{0}\circ \psi +
\psi_{*}(\nabla \psi_{0})\Big)
$$ are two
lightlike normal vector fields with $\langle \xi, \eta \rangle =
1$.
\end{lemma}
\begin{proof} It is clear that $\xi$ is a lightlike normal vector field.
Let $T\in \mathfrak{X}^{\perp}(M^{2})$ be the normal component of
$\partial_0$. A direct computation shows
\begin{equation}\label{laT}
T=\partial_{0}\circ \psi +\psi_{*}(\nabla \psi_{0}).
\end{equation}
Now, taking into account that $\xi$ and $T$ span the normal bundle
of $M^2$ and $\langle \xi, T\rangle=-\psi_{0}$, $\langle T,
T\rangle=-1-\|\nabla \psi_{0}\|^2$, we deduce the formula for
$\eta $.
\end{proof}
\begin{remark}\label{importante}{\rm (a) Consequently, a spacelike surface in $\L^4$
which factors through $\Lambda^{+}$ must be
orientable. (b) On the other hand, taking into account $\nabla^{\perp}\xi=0$,
the normal vector field $\eta$ in Lemma \ref{campos}
also satisfies $\nabla^{\perp}\eta=0$. Thereofore, the normal connection of $\psi$ must be flat.}
\end{remark}

\begin{proposition}\label{weingarten}
Suppose that $\psi:M^{2}\rightarrow \L^4$ is a spacelike surface
which factors through $\Lambda^{+}$ and let $\xi, \eta$ be
as above. Then, we have
$$A_{\xi}=-I \,\,\quad \textrm{and}\,\,\quad
A_{\eta}= -\frac{1+\| \nabla \psi_{0}\|^2
}{2\,\psi_{0}^{2}}I+\frac{1}{\psi_{0}}\,\nabla^{2}
\psi_{0},$$ where $\nabla^{2} \psi_{0} (v)=\nabla_{v}(\nabla
\psi_{0})$, for every $v\in T_{q}M^{2}$, $q\in M^2$.
\end{proposition}
\begin{proof} Clearly, we have $\overline{\nabla}_v \psi=\psi_{*}(v)$. Therefore
$\nabla^{\perp}_{v}\psi=0$, and the Weingarten formula
directly gives $A_{\xi}=-I$. On the other hand, since $$ \overline{\nabla}_{v}T=\overline{\nabla}_{v}
\big(\psi_{*}(\nabla
\psi_{0})\big)=\psi_{*}(\nabla_{v}\nabla\psi_{0})+\mathrm{II}\big(v,(\nabla
\psi_{0})_{q}\big),
$$
we get  $A_{T}=-\nabla^{2} \psi_{0}$. Now, the formula for
$A_{\eta}$ follows from Lemma \ref{campos}.
\end{proof}

\begin{remark}\label{33}{\rm The previous result implies that a spacelike
surface $M^2$ in $\L^4$ which factors through a lightcone has no point where $\mathbf{H}$ vanishes (see (\ref{curme})), and $M^2$ is totally umbilical if and and only if $\eta$ is umbilical. 
}
\end{remark}

A direct computation from Proposition \ref{weingarten} shows,
\begin{corollary}\label{umbi}
A spacelike surface $\psi:M^{2}\rightarrow \L^4$
which factors through a lightcone is totally umbilical in $\L^{4}$ if and only if 
$$\nabla^{2} \psi_{0}=\frac{1}{2}\triangle \psi_{0}\,I.$$
\end{corollary}\hfill{$\Box$}

Now, from (\ref{curme}) and (\ref{101}) the following formulae for the mean curvature vector field $\mathbf{H}$ and the Gauss curvature are obtained.

\begin{corollary}\label{H}
Suppose that $\psi : M^{2}\rightarrow \L^4$ is a spacelike surface
which factors through $\Lambda^{+}$. Then,
\begin{equation}\label{mean_curvature}
\mathbf{H}=\left(\,\frac{\bigtriangleup
\psi_{0}}{2\,\psi_{0}}-\frac{1+\| \nabla
\psi_{0}\|^2}{\psi_{0}^{2}}\,\right)\,
\xi+\frac{1}{\psi_{0}}\,T
\end{equation}
\begin{equation}\label{Gauss_curvature}
K=\frac{1+\| \nabla
\psi_{0}\|^2}{\psi_{0}^2}-\frac{\bigtriangleup
\psi_{0}}{\psi_{0}},
\end{equation}
\noindent and therefore,
\begin{equation}\label{Gauss_curvature2}
K=\langle \mathbf{H},\mathbf{H}\rangle.
\end{equation}
\end{corollary}

\begin{remark}
{\rm In the terminology of \cite[Def. 1.1]{Liu}, the function $\frac{1}{2}\langle \triangle \psi ,\eta \rangle$ is called the mean curvature of the spacelike surface in a lightcone. Using the well-known Beltrami formula $\triangle \psi=
2\mathbf{H}$ and 
\begin{equation}\label{patata}
K=-\mathrm{tr}A_{\eta},
\end{equation}
which follows from (\ref{Gauss_curvature}) and (\ref{curme}), we get $\frac{1}{2}\langle \triangle \psi ,\eta \rangle=-\langle \mathbf{H},\mathbf{H}\rangle$.
Note now that Proposition
\ref{weingarten} implies that, for a spacelike surface in a lightcone, (\ref{curme}) reduces to,
\begin{equation}\label{h}
\mathbf{H}=-\frac{1}{2}K\xi-\eta.
\end{equation}
As a direct consequence of the previous formula we have,
a spacelike surface in a lightcone is pseudo-umbilical if and only if it is totally umbilical.}					
\end{remark}

\begin{remark}
{\rm A direct computation from Corollary \ref{H} shows,
\begin{equation}\label{cafe}
K=-\triangle \log \psi_{0}+\frac{1}{\psi_{0}^{2}}.
\end{equation}
This formula has a nice meaning, the new metric on $M^2$ defined by $\psi_0^{-2} \langle\,\,,\,\,\rangle$ has constant Gauss curvature $1$. This fact has an obvious topological consequence: if a $2$-dimensional manifold $S$ admits a spacelike  immersion in a lightcone of $\L^4$ and $S$ is compact, then from the Gauss-Bonnet theorem and Remark \ref{importante}(a), we have that $S$ must be homeomorphic to $\S^2$. This can be also deduced from a direct topological argument in Proposition \ref{esf}. On the other hand, the well-known uniformization theorem implies that every simply connected two dimensional Riemannian manifold is conformally embedded in the unit sphere $\S^{2}$. This property has been used to show that every simply connected two dimensional Riemannian manifold admits an isometric embedding in the lightcone $\Lambda^{+}$ of $\L^{4}$ \cite{LUY}.
}
\end{remark}

Directly from (\ref{h}) and (\ref{Gauss_curvature2}) we get \cite[Th. 4.3]{chen1}.

\begin{corollary}\label{noname1}
Let $\psi : M^{2}\rightarrow \L^4$ be a spacelike surface
which factors through a lightcone. Then the following conditions are equivalent,
\begin{enumerate}
\item $K$ is a constant.
\item The mean curvature vector field satisifies $\nabla^{\perp}\mathbf{H}=0$.
\end{enumerate}
\end{corollary}

The following result shows how the sign of the Gauss curvature influences on
the existence of relative extrems of the function $\psi_0$.

\begin{proposition}\label{noname}
Let $\psi:M^{2}\rightarrow \L^4$ be a spacelike surface
which factors through a future (resp. past) lightcone. Assume $K\leq 0$, then the function $\psi_{0}$ attains no local maximum (resp. minimum) value. 
\end{proposition}
\begin{proof}
On the contrary, if there exists a local maximum point $q\in
M^{2}$ of $\psi_{0}$ then from (\ref{Gauss_curvature}) the Gauss curvature satisfies $K>0$ on a neighbourhood of $q$. The
proof for past lightcones works in a similar way.
\end{proof}

\section{Several examples}\label{examples}

This section is devoted to describe two families of spacelike surfaces in $\L^{4}$ through the lightcone $\Lambda^{+}$. 
Let $\psi:M^{2}\to \L^{4}$ be a spacelike immersion with $\psi(M^{2})\subset \Lambda^{+}$. Now, for each $\sigma \in C^{\infty}(M^{2})$, we construct a new spacelike surface in $\Lambda^{+}$ taking $\psi_{\sigma}=e^{\sigma}\,\psi$. If we denote $g=\psi^{*}\langle\,
,\, \rangle$ then, we put
$g_{\sigma}=\psi_{\sigma}^{*}\langle\, ,\,
\rangle=e^{2\sigma}g$ (compare with \cite[Prop. 7.5]{Da}). Therefore, we have,
\begin{equation}\label{ultima}
K_{\sigma}=\frac{K-\triangle \sigma}{e^{2\sigma}},
\end{equation}
where $K$ and $K_{\sigma}$ are the Gauss curvature of $g$ and $g_{\sigma}$ respectively.

\vspace{0.5cm}
\noindent {\bf Example 1}.
Consider the following spacelike immersion,
$$
\psi(x,y)=\big(\cosh x,\sinh x,\cos y,\sin y\big),\,\,\,(x,y)\in \R^{2}.
$$
It is not difficult to show that $\psi$ is an isometric immersion
from the Euclidean plane $\E^{2}$ in $\L^{4}$ through the lightcone $\Lambda^+$. From (\ref{h}) it follows that $\mathbf{H}$ is lightlike everywhere.
That is, $\psi$ is a marginally trapped surface in the terminology of \cite{MS}.

For each $\sigma \in C^{\infty}(\R^{2})$, we will identify with a superscript $\sigma$ the differential operators associate to the metric $g_{\sigma}$. Direct computations show that,
$$
\nabla^{\sigma}\psi_{0}=\frac{1}{e^{\sigma}}\Big[(\sigma_{x}\cosh x +\sinh x)\partial_{x}+(\sigma_{y}\cosh x) \partial_{y}\Big],
$$ 
$$
\triangle^{\sigma} \psi_{0}=
\frac{1}{e^{\sigma}}\Big[(1+\triangle^{0} \sigma +\|\nabla^{0}\sigma\|_{0}^{2})\cosh x+(2\sigma_{x})\sinh x\Big],
$$
where $\|\,\,\|_{0}$, $\triangle^{0}$ and $\nabla^{0}$ denote the corresponding Euclidean operators.
Note that the above formulae permit to deduce $K_{\sigma}=-\triangle^{0} \sigma/e^{2\sigma}$ from (\ref{Gauss_curvature}).
Therefore, we get that,
$$
(\nabla^{\sigma})^{2}_{\partial_{x}}\psi_{0}=\frac{1}{e^{\sigma}}\Big[\Big((1+\sigma_{y}^{2}+\sigma_{xx})\cosh x
+\sigma_{x}\sinh x \Big)\partial_{x}+\Big((\sigma_{xy}-\sigma_{x}\sigma_{y})\cosh x\Big)\partial_{y}\Big],
$$
$$
(\nabla^{\sigma})^{2}_{\partial_{y}}\psi_{0}=\frac{1}{e^{\sigma}}\Big[\Big((\sigma_{xy}-\sigma_{x}\sigma_{y})\cosh x\Big)\partial_{x}+\Big((\sigma_{x}^{2}+\sigma_{yy})\cosh x+\sigma_{x}\sinh x\Big)\partial_{y}\Big].
$$
Proposition \ref{weingarten} can be claimed to deduce that, with respect to the basis $\{\partial_{x},\partial_{y}  \}$, the Weingarten endomorphism associated to the corresponding normal lightlike section $\eta_{\sigma}$ is characterized by,
$$
\mathrm{II}_{\eta_{\sigma}}=\frac{1}{2}
\begin{pmatrix}
    \sigma_{x}^{2}-\sigma_{y}^{2}-2\sigma_{xx}-1  & 2(\sigma_{x}\sigma_{y}-\sigma_{xy})   \\
                                  &                    \\
   2(\sigma_{x}\sigma_{y}-\sigma_{xy}) & \sigma_{y}^{2}-\sigma_{x}^{2}-2\sigma_{yy}+1
\end{pmatrix}.
$$
Corolary \ref{umbi} implies that the immersion $\psi_{\sigma}$ is totally umbilical if and only if $\sigma$ satisfies the following system of partial differential equations,
\begin{equation}\label{PDE}
\sigma_{x}^{2}-\sigma_{y}^{2}-\sigma_{xx}+\sigma_{yy}=1,\,\,\,\sigma_{xy}=\sigma_{x}\sigma_{y}.
\end{equation}
The following table shows several solutions of (\ref{PDE}) and the Gauss curvature of $\psi_{\sigma}$ on the corresponding open subset of $\R^{2}$. 

\begin{table}[hline]
\begin{center}
\begin{tabular}{|c|c|}\hline
Solution  $e^{\sigma}$ & Gauss curvature $K_{\sigma}$\\  
\hline
$e^{x}$ & $0$\\
$a\,\mathrm{sech}\, (x)$ & $1/a^{2}$  \\
$a\,\mathrm{cosech}\, (x)$ & $-1/a^{2}$ \\
$e^{x}/(e^{2x}-1)$ & $-4$ \\
$a\sec \,(y)$ & $-1/a^{2}$ \\
$a\,\mathrm{cosec} \,(y)$ & $-1/a^{2}$\\
\hline
\end{tabular} 
\end{center}
\label{Gauss}
\end{table}

\noindent Note that, if we replace $x$ with $y$ in the solutions $\sigma$ of the above table, we obtain immersions $\psi_{\sigma}$ with Gauss constant curvature which are not totally umbilical.

\vspace{0.5cm}
\noindent {\bf Example 2}. Consider now the following embedding, 
$$\psi:\S^{2}\rightarrow
\Lambda^{+},\quad \psi(x,y,z)=(1,x,y,z).$$
Clearly $\psi$ is
a totally umbilical spacelike surface with mean curvature vector
field $\mathbf{H}=-\psi+\partial_{0}\circ \psi.$ It is not difficult to see that the induced metric is the usual one, $\langle\,\,\, ,\,\,\,\rangle_{0}$, of constant Gauss curvature $1$.

For each $\sigma \in C^{\infty}(\S^2)$,
we denote with the superscript $\sigma$ the geometric operators associated to $g_{\sigma}=e^{2\sigma}\langle\,\,\,,\,\,\,\rangle_{0}$. The Levi-Civita connection of $g_{\sigma}$ satisfies, 
$$
\nabla^{\sigma}_{X}Y=\nabla^{0}_{X}Y-\langle X,Y \rangle_{0}\nabla^{0}\sigma+(X\sigma)Y+(Y\sigma)X+(1+\mathbf{P}\sigma)\langle X,Y \rangle_{0} \mathbf{P},
$$
for every $X,Y\in \mathfrak{X}(\S^{2})$.
Here we denote by the superscript $0$ the differential operators of $\E^{3}$ and $\mathbf{P}$ for the position vector field. Now, from Proposition \ref{weingarten} we get,
\begin{equation}\label{umbisigma}
\mathrm{II}_{\eta_{\sigma}}=\frac{1}{2}\left[(1+\mathbf{P}\sigma)^{2}-\|\nabla^{0}\sigma\|^{2}_{0}\right]\langle\,\,\, ,\,\,\,\rangle_{0}-\mathrm{Hess}^{0}(\sigma)+d\sigma\otimes d\sigma.
\end{equation}

We end this section showing an application of the previous formula.
Let $u \in \L^4$ be a vector which satisfies $\langle u, u\rangle =-1$ and $u_0<0$,
and let $r$ be a positive real number. Put
$$\S^2 (u,r)=\big\{x\in \L^4 \, : \,\langle x, x\rangle =0, \quad
\langle u, x\rangle = r \big\}.$$
The surface $\S^2(u,r)$ may be parametrized by
$
\psi_{\sigma}=e^{\sigma}\psi$ where $\sigma=\log r-\log \langle u, \psi\rangle$. In this case $\mathrm{Hess}^{0}(\sigma)=d\sigma\otimes d\sigma$ and formula (\ref{umbisigma}) reduces to $\mathrm{II}_{\eta_{\sigma}}=(1/2r^{2})g_{\sigma}$. Therefore, $A_{\eta_{\sigma}}=-\frac{1}{2r^2}\,I$ and
$K_{\sigma}=\frac{1}{r^{2}}$. In particular, the surfaces $\S^2(u,r)$ are totally umbilical.
Conversely, if $\psi : \S^{2}\rightarrow \L^4$ is a totally umbilical spacelike immersion 
which factors through the lightcone $\Lambda^{+}$, then $A_{\eta}=\frac{-1}{2r^2} \,I$ where $K=1/r^{2}$. Moreover, $w=\frac{-1}{2r^2}\psi+\eta$ is timelike and constant in $\L^4$. Now it is not difficult to show that,
$$
\psi(M^{2})=\S^{2}(u,r),
$$
where $u=rw$.
We will refer to $\S^{2}(u,r)$ as the totally umbilical round spheres of the lightcone $\Lambda^{+}$.

\section{Compact spacelike surfaces in a lightcone}

\begin{proposition}\label{esf}
Every compact spacelike surface in $\L^4$ which factors through a
lightcone is a topological $2$-sphere.
\end{proposition}
\begin{proof} Let $\psi \colon M^{2}\rightarrow \L^4$ be a compact spacelike immersion
with $\psi(M^{2})\subset \Lambda^{+}$. Consider the map $F :
M^{2}\rightarrow \S^2$ given by $F = \pi \circ \alpha \circ
\psi$ where $\pi:(0,+\infty)\times \S^{2}\rightarrow \S^{2}$
is the projection onto the second factor and $\alpha : \Lambda^+
\rightarrow (0,+\infty)\times \S^{2}$ is the diffeomorphism
defined by
$$
\alpha(v)=\Big(v_{0} , \frac{1}{v_{0}}(v_{1},v_{2},v_{3})\Big).
$$
The map $F$ is a local diffeomorphism. The compactness of $M^2$
and the connectedness of $\S^2$ imply that $F$ is a covering map.
Finally, the simply connectedness of $\S^2$ gives $F$ is a
diffeomorphism (see \cite[Prop. 5.6.1]{DoCarmo} for details). The
proof for past lightcones works in a similar way.
\end{proof}

\begin{remark}
{\rm It should be noted that the same argument as in the previous result shows that every compact $(n\geq 2)$-dimensional submanifold in $\L^{n+2}$ which factors through a lightcone is a topological $n$-sphere.
However, that is not the case if the codimension of the spacelike submanifold is assumed to be $\geq 3$. 
}
\end{remark}

\begin{remark}{\label{will}}
{\rm Proposition \ref{esf} and (\ref{Gauss_curvature2}) allow, making use of the Gauss-Bonnet theorem, to get
$$
\int_{M^2}\langle \mathbf{H},\mathbf{H}\rangle dA=4\pi
$$
for any compact spacelike surface $M^2$ in a lightcone of $\L^4$. Notice that the integrand may be negative somewhere as well-known examples show. 

For a general compact spacelike surface with $\mathbf{H}$ non-zero everywhere the existence of some point $p$ where $\langle \mathbf{H},\mathbf{H}\rangle (p)>0$ was already known \cite[Rem. 4.2]{AER}. On the other hand, noncompact spacelike surfaces in a lightcone of $\L^4$ such that $\mathbf{H}$ is lightlike everywhere are shown to exist in Section 4.}
\end{remark}

\begin{theorem}\label{complejo}
Let $\psi : M^{2}\rightarrow \L^4$ be a complete spacelike surface
which factors through the lightcone $\Lambda^{+}$. Assume $K$ is constant. If $\psi_{0}$ attains a local maximum value, then $M^{2}$ is a totally umbilical round sphere.
\end{theorem}
\begin{proof}
From Proposition \ref{noname} we have $K>0$. The classical Myers theorem is claimed now to get $M^{2}$ is compact. Therefore, from Proposition \ref{esf}, $M^{2}$ is isometric to a sphere of Gauss curvature $K$. 
Consider now $\Omega$ the quadratic differential on $M^{2}$ locally
given by
$$
\Omega=\langle \mathrm{II}(\partial_{z},\partial_{z}),\eta
\rangle dz^{2},
$$
where $z=x+\mathbf{i}y$ and $(x,y)$ are local isothermal
parameters on $M^{2}$ with $\langle
\partial_{x},\partial_{x}\rangle=\langle
\partial_{y},\partial_{y}\rangle=F>0$ . Then $\Omega$ is well
defined and $\Omega=0$ if and only if $M^{2}$ is totally umbilical
(see for example \cite[Sect. 2]{Hoff}). Now, from the Codazzi
equation it follows that,
$$
\nabla^{\perp}_{\partial_{\bar{z}}}\mathrm{II}(\partial_{z},
\partial_{z})=\nabla^{\perp}_{\partial_{z}}\mathrm{II}(\partial_{z},
\partial_{\bar{z}})-\frac{1}{F}\frac{\partial F}{\partial z}\,\mathrm{II}
(\partial_{z},\partial_{\bar{z}})=\frac{1}{2}F\,\nabla^{\perp}_{\partial_{z}}\mathbf{H}.
$$
Using $\nabla^{\perp}\eta=0$ we get,
$\partial_{\bar{z}}\langle\mathrm{II}(\partial_{z},\partial_{z}),
\eta\rangle=\frac{1}{2}F\,\langle
\nabla^{\perp}_{\partial_{z}}\mathbf{H}, \eta \rangle
=\frac{1}{2}F\,\partial_{z}\langle\mathbf{H}, \eta\rangle$.
Therefore, $\Omega$ is holomorphic if and only if the function
$\langle\mathbf{H}, \eta\rangle$ is constant.
But this is the case because $\langle
\mathbf{H},\eta \rangle =-K/2$. Consequently,  being
$M^{2}$ a topological sphere, it follows that $\Omega=0$.
\end{proof}

\begin{remark}
{\rm Under the assumption of Theorem \ref{complejo}, we know from Corollary \ref{noname1} that the mean curvature vector field is parallel. As soon as we know that $M^2$ is a topological sphere the proof follows now from \cite[Cor. 4.5]{AER}
}
\end{remark}

\begin{remark}\label{contraejemplo}
{\rm Of course there exist non totally umbilical isometric immersions of the unit round sphere $\S^{2}$ in $\L^{4}$. For instance, $\psi:\S^{2}\to \L^{4}$ given by $\psi(x,y,z)=(\cosh x,\sinh x ,y ,z)$. In fact, it is not difficult that $\mathbf{N_{1}}$ and $\mathbf{N_{2}}$ given by,
$$
\mathbf{N_{1}}(x,y,z)=(\cosh x, \sinh x ,0,0),\quad \mathbf{N_{2}}(x,y,z)=(x\sinh x, x\cosh x ,y,z),
$$
$(x,y,z)\in \S^{2}$, are normal vector fields such that $\langle \mathbf{N_{1}}, \mathbf{N_{1}} \rangle =-\langle \mathbf{N_{2}}, \mathbf{N_{2}} \rangle=-1$, $\langle \mathbf{N_{1}}, \mathbf{N_{2}} \rangle=0$ and the corresponding shape operators satisfy,
$$
A_{\mathbf{N_{1}}}(v_{1},v_{2},v_{3})=v_{1}\,(x^{2}-1,xy,xz),\quad A_{\mathbf{N_{2}}}(v_{1},v_{2},v_{3})=-(v_{1},v_{2},v_{3}),$$
for all $(v_{1},v_{2},v_{3})\in T_{(x,y,z)}\S^{2}$ and $(x,y,z)\in \S^{2}$.
}
\end{remark}

\begin{remark}
{\rm Previous formula (\ref{cafe}) may be generalized as follows:
$$
K=-\triangle \log \langle \psi, u\rangle+\frac{1}{\langle \psi, u\rangle^{2}},
$$
where $u\in \L^{4}$ satisfies $\langle u,u\rangle=-1$, $u_{0}<0$. In particular, this gives in the compact case,
\begin{equation}\label{leche}
\int_{\S^{2}}\frac{1}{\langle \psi, u\rangle^{2}}\,dA=4\pi.
\end{equation}}
\end{remark}

A direct consequence of formula $(\ref{leche})$ and Schwartz inequality gives.

\begin{proposition}\label{area}
Let $\psi : \S^{2}\rightarrow \L^4$ be a spacelike immersion
which factors through $\Lambda^{+}$. Then, for every $u \in \L^4$ which satisfies $\langle u, u\rangle =-1$ with $u_{0}<0$, we have the following upper bound for the area of the induced metric,
$$
\mathrm{area}(\S^{2},\langle\,\, ,\,\,\rangle)\leq 2\sqrt{\pi}\,\|\langle \psi, u\rangle\|,
$$
where $\|\,\,\|$ denotes the $L^{2}$ norm. The equality holds for some $u$ if and only if the surface is the 
totally umbilical round sphere $\S^{2}(u,r)$, $r=\langle \psi,u\rangle\in \R^{+}$.
\end{proposition}

If $\psi : \S^{2}\rightarrow \L^4$ is a spacelike immersion
which factors through a lightcone, denote by $0=\lambda_{0}< \lambda_{1}\leq \lambda_{2}\leq ...$ the spectrum of the Laplace operator of the induced metric. The Hersch inequality \cite{Hersch} states that,
\begin{equation}\label{hersch}
\lambda_{1}\leq \frac{8\pi}{\mathrm{area}(\S^{2},\langle\,\, ,\,\,\rangle)},
\end{equation}
and the equality holds if and only if $(\S^{2},\langle\,\, ,\,\,\rangle)$ has constant Gauss curvature.

\begin{remark}\label{Reilly}
{\rm The Hersch inequality, taking into account (\ref{Gauss_curvature2}) and the Gauss-Bonnet formula, may be rewritten for a compact spacelike surface in $\Lambda^{+}$ as follows,
\begin{equation}\label{Re}
\lambda_{1}\leq 2\,\frac{\int_{\S^{2}}\langle \mathbf{H},\mathbf{H}\rangle\,dA}{\mathrm{area}(\S^{2},\langle\,\, ,\,\,\rangle)},
\end{equation}
which looks like formally equal to the well-known Reilly extrinsic above bound for $\lambda_1$, in the case of a compact surface in $m$-dimensional Euclidean space \cite{Re}. However, the Reilly inequality does not hold for a compact spacelike surface in $\L^4$ in general. As a counter-example consider the isometric immersion $\psi:\S^{2}\to \L^{4}$ given in Remark \ref{contraejemplo}. The mean curvature vector field is $\mathbf{H}_{(x,y,z)}=-\frac{1}{2}(x^2-1)\mathbf{N_{1}}-\mathbf{N_{2}}$, $(x,y,z)\in \S^{2}$. Hence $\langle \mathbf{H}, \mathbf{H}\rangle=-\frac{1}{4}(x^2-1)^{2}+1$ and therefore,
$$
\int_{\S^{2}}\langle \mathbf{H}, \mathbf{H}\rangle\, dA< 4\pi.
$$
}

\end{remark}

\begin{theorem}\label{calor}
Let $\psi : \S^{2}\rightarrow \L^4$ be a spacelike immersion
which factors through $\Lambda^{+}$. Then, for every $u \in \L^4$ which satisfies $\langle u, u\rangle =-1$, $u_{0}<0$, we have,
$$
\lambda_{1}\leq \frac{2}{\mathrm{min}\,\langle \psi, u\rangle^{2}},
$$
and the equality holds for some $u$ if and only if  the surface is the 
totally umbilical round sphere $\S^{2}(u,r)$, $r=\langle \psi,u\rangle\in \R^{+}$.
\end{theorem}
\begin{proof}
The announced inequality is directly deduced from (\ref{leche}) and (\ref{hersch}). 
Assume now the equality holds for the timelike vector $u$. The Hersch inequality gives that $(\S^{2},\langle\,\, ,\,\,\rangle)$ has constant Gauss curvature and therefore the result is followed from Theorem \ref{complejo}. Conversely, since $\S^{2}(u,r)$ has constant Gauss curvature $1/r^2$ we have $\lambda_{1}=2/r^{2}$.
\end{proof}

\begin{remark}
{\rm As a particular case of Theorem \ref{calor}, we have that for every compact spacelike immersion $\psi$ which factors through $\Lambda^{+}$, $$\lambda_{1}\leq \frac{2}{\mathrm{min}\,\psi_{0}^{2}},$$ and the equality holds if and only if $\psi_{0}$ is a constant.
}
\end{remark}


\begin{thebibliography}{999}

\bibitem{AAR} J.A. Aledo, Luis J. Al\'{\i}as and A. Romero,
A new proof of Liebmann classical rigidity theorem for surfaces in
space forms, {\it Rocky Mt. J. Math.}, {\bf 35} (2005),
1811--1824.


\bibitem{AER} L. J. Al\'{\i}as, F. J.M. Estudillo and A.
Romero, Spacelike submanifolds with parallel mean curvature in
pseudo-Riemannian space forms, {\it Tsukuba J. Math.}, {\bf 21}
(1997), 169--179.




\bibitem{AR} J.A. Aledo and A. Romero, Compact spacelike surfaces
in the $3$-dimensional de Sitter space with non-degenerate second
fundamental form, {\it Differ. Geom. Appl.}, {\bf 19} (2003),
97--111.

\bibitem{AD} A. Asperti and M. Dajczer, Conformally flat Riemannian
manifolds as hypersurfaces of the light cone, {\it Can. Math.
Bull.}, {\bf 32} (1989), 281--285.




\bibitem{chen1} B.Y. Chen, Classification of spatial surfaces with parallel mean curvature vector in pseudo-Euclidean spaces of arbitrary dimension, {\it J. Math. Phys.}, {\bf 50}
(2009), 043503(1-14).




\bibitem{Da} M. Dajczer, {\it Submanifolds and isometric immersions}, Math. Lectures Series {\bf 13}, 
Publish or Perish, Houston, 1990.

\bibitem{DoCarmo} M.P. Do Carmo, {\it Differential Geometry of Curves and Surfaces},
Prentice-Hall, Englewood Cliffs, New Jersey, 1976.



\bibitem{Hersch} J. Hersch, Quatre propriétés isopérimétriques de membranes sphériques homogènes. {\it C. R. Acad. Sci. Paris Sér. A}, {\bf 270}
(1970), 1645--1648.




\bibitem{Hoff} D.A. Hoffman, Surfaces of constant mean curvature in manifolds
of constant curvature, {\it J. Differential Geom.}, {\bf 8}
(1973), 161--176.


\bibitem{IPR} S. Izumiya, D. Pei and M.C. Romero Fuster, Umbilicity of space-like submanifolds of Minkowski space, {\it Proc. Roy. Soc. Edin. A}, {\bf 134} (2004), 375--387.



\bibitem{LUY} H. Liu, M. Umehara and K. Yamada, The duality of conformally flat manifolds, {\it Bull. Braz. Math. Soc.}, {\bf 42} (2011), (no. 1),131--152.



\bibitem{Liu} H.L. Liu, Surfaces in lightlike cone, {\it J. Math. Anal. Appl.},
{\bf 325} (2007), 1171--1181.


\bibitem{LiuJung} H.L.Liu and S. D. Jung, Hypersurfaces in lightlike cone, {\it J. Geom. Phys.}, {\bf 58} (2008), 913--922.


\bibitem{Mag} M.A. Magid, Isometric inmersions of Lorentz space with parallel
second fundamental forms, {\it Tsukuba J. Math.}, {\bf 8} (1984), 31--54.

\bibitem{MS} M. Mars and J.M.M. Senovilla, Trapped surfaces and symmetries, {\it Class. Quantum Grav.}, {\bf 20} (2003), 1293--1300. 

\bibitem{Re} R.C. Reilly, On the first eigenvalue of the Laplacian for compact submanifolds of Euclidean space, {\it Comment. Mat. Helvetici}, {\bf 52}(1977), 525--533.



\end{thebibliography}
\end{document}